\documentclass[11pt,reqno]{amsart}
\usepackage[T1]{fontenc}
\usepackage{lmodern}
\usepackage[a4paper,margin=26mm,headheight=14pt]{geometry}
\usepackage{amsmath,amssymb,mathtools,mathrsfs,booktabs,array}
\usepackage{microtype}
\usepackage{aliascnt}
\usepackage[hidelinks,unicode]{hyperref}
\usepackage[nameinlink,noabbrev]{cleveref}
\hypersetup{pdftitle={Graph-Diameter Criteria and Sharp Dissipation in Discrete Wasserstein--Fisher Geometry},pdfsubject={Boundary entropy Hessians, weighted graphs, and nonlinear probability flows},pdfauthor={Rubin Zhang and Ning Zhang}}
\newtheorem{theorem}{Theorem}[section]
\newaliascnt{lemma}{theorem}
\newtheorem{lemma}[lemma]{Lemma}
\aliascntresetthe{lemma}
\newaliascnt{proposition}{theorem}
\newtheorem{proposition}[proposition]{Proposition}
\aliascntresetthe{proposition}
\newaliascnt{corollary}{theorem}
\newtheorem{corollary}[corollary]{Corollary}
\aliascntresetthe{corollary}
\theoremstyle{definition}
\newaliascnt{definition}{theorem}

\aliascntresetthe{definition}
\theoremstyle{remark}
\newaliascnt{remark}{theorem}
\newtheorem{remark}[remark]{Remark}
\aliascntresetthe{remark}
\crefname{theorem}{Theorem}{Theorems}
\crefname{lemma}{Lemma}{Lemmas}
\crefname{proposition}{Proposition}{Propositions}
\crefname{corollary}{Corollary}{Corollaries}
\crefname{remark}{Remark}{Remarks}
\crefname{section}{Section}{Sections}
\crefname{table}{Table}{Tables}
\numberwithin{equation}{section}
\newcommand{\R}{\mathbb R}
\newcommand{\one}{\mathbf 1}
\newcommand{\Pplus}{\mathcal P_+}
\newcommand{\Pbar}{\overline{\mathcal P}}
\newcommand{\A}{\mathcal A}
\newcommand{\B}{\mathcal B}
\newcommand{\PP}{\mathcal I}
\newcommand{\Q}{\mathcal Q}
\newcommand{\dd}{\,\mathrm d}
\DeclareMathOperator{\diag}{diag}
\DeclareMathOperator{\Var}{Var}
\DeclareMathOperator{\osc}{osc}
\DeclareMathOperator{\dist}{dist}
\DeclareMathOperator{\diam}{diam}
\DeclareMathOperator{\supp}{supp}
\DeclareMathOperator{\Hess}{Hess}

\newcommand{\AW}{\A_W}
\newcommand{\AC}{\A_C}
\newcommand{\IW}{\mathscr I_W}
\newcommand{\IF}{\mathscr I_F}
\newcommand{\doi}[1]{\href{https://doi.org/#1}{\nolinkurl{#1}}}
\newcommand{\arxiv}[1]{\href{https://arxiv.org/abs/#1}{arXiv:\nolinkurl{#1}}}
\allowdisplaybreaks[1]
\title[Graph diameter in discrete Wasserstein--Fisher geometry]{Graph-Diameter Criteria and Sharp Dissipation\\in Discrete Wasserstein--Fisher Geometry}
\author[Rubin Zhang]{Rubin Zhang}
\address{School of Mathematics and Statistics, Huazhong University of Science and Technology, Wuhan, Hubei 430074, China; and Department of Mathematical and Statistical Sciences, University of Alberta, Edmonton, Alberta T6G 2G1, Canada}
\email{rbzhang3@163.com}
\thanks{Rubin Zhang was supported by the China Scholarship Council (CSC)}
\author[Ning Zhang]{Ning Zhang}
\address{School of Mathematics and Statistics, Huazhong University of Science and Technology, Wuhan, Hubei 430074, China}
\email{nzhang2@hust.edu.cn}
\thanks{Corresponding author: Ning Zhang. Ning Zhang was supported by the National Natural Science Foundation of China (Grant No.~62535007) and Huazhong University of Science and Technology (Grant No.~5003011074)}
\date{October 24, 2026}
\begin{document}
\begin{abstract}
We study the probability simplex of a finite connected weighted graph with an Onsager operator combining logarithmic-mean transport and mass-conserving Fisher reaction. For every fixed positive reaction parameter, we prove that the entropy Hessian has a finite global lower bound if and only if the original transport graph has diameter at most two. Near a boundary probability, the corresponding criterion is that its positive support reaches every vertex in at most two steps. The proof controls all potential directions under arbitrary multiscale boundary degeneration. For graphs of diameter at most two, the optimal lower bound is Lipschitz on bounded parameter intervals, including the pure-transport endpoint. We also determine the largest instantaneous logarithmic-amplitude dissipation function depending only on the graph and the current amplitude, and construct two-level probabilities attaining it. Its positivity has the same two-step criterion. As consequences, we obtain graph-dependent entropy-decay estimates and exact projective contraction bounds without assuming nonnegative entropic curvature.
\end{abstract}
\maketitle

\section{Introduction and main results}\label{sec:intro}

The interpretation of finite-state Markov evolution as an entropy gradient flow gives a geometric approach to discrete probability \cite{Maas2011,ErbarMaas2012,Mielke2013}. In the logarithmic-mean geometry, entropy convexity is measured by a Riemannian Hessian on the interior of the probability simplex. Adding a mass-conserving Fisher reaction preserves the gradient-flow structure, but changes the behavior of this Hessian near the boundary. We use ``Wasserstein--Fisher geometry'' for the specific graph-based Onsager sum in \eqref{eq:operator}; its reaction is of Fisher--Rao type, while its transport uses the logarithmic mean on the original graph. Although the added covariance couples every pair of positive coordinates, the original transport graph remains visible in the boundary singularities.

This paper identifies the precise graph obstruction. For each fixed positive reaction strength, finite global entropic curvature is equivalent to diameter at most two. The corresponding local statement is a two-step reachability condition on a boundary support. A second, dynamical calculation identifies exactly the same graph condition in the optimal instantaneous decay of the logarithmic amplitude. These statements distinguish three properties that should not be conflated: a finite Hessian lower bound, a nonnegative Hessian lower bound, and convergence of the entropy flow.

\subsection{The model}
Throughout the paper, $G=(V,E,w)$ is a finite connected undirected graph with $N=|V|\ge2$. Its edge weights satisfy $w_{ij}=w_{ji}>0$ for $\{i,j\}\in E$, and $w_{ij}=0$ otherwise; also $w_{ii}=0$. Every sum over $E$ counts an unoriented edge once. Graph distances refer to this original transport graph and disregard the magnitudes of its positive weights. Write
\[
 \Pplus=\{\rho\in(0,\infty)^N:\one^T\rho=1\},\qquad
 \Pbar=\{\rho\in[0,\infty)^N:\one^T\rho=1\}.
\]
The reference probability is uniform, $\pi_i=1/N$. The logarithmic mean is
\begin{equation}\label{eq:logmean}
 \theta(a,b)=\int_0^1 a^t b^{1-t}\dd t
 =\frac{a-b}{\log a-\log b},\qquad a,b>0,
\end{equation}
with the continuous value $\theta(a,a)=a$. We write $\theta_1(a,b)=\partial_a\theta(a,b)$ and $\theta_2(a,b)=\partial_b\theta(a,b)$ for its first partial derivatives. For $\rho\in\Pplus$, define
\begin{align}
 W(\rho)&=\sum_{\{i,j\}\in E}w_{ij}\theta(\rho_i,\rho_j)(e_i-e_j)(e_i-e_j)^T,\notag\\
 C_\rho&=\diag\rho-\rho\rho^T,\qquad
 K_\lambda(\rho)=W(\rho)+\lambda C_\rho,\quad\lambda\ge0.\label{eq:operator}
\end{align}
The restriction of $K_\lambda$ to $\one^\perp$ is positive definite. It defines the Riemannian metric
\begin{equation}\label{eq:metric}
 g_{\lambda,\rho}(\sigma,\tau)=\sigma^TK_\lambda(\rho)^\dagger\tau,
 \qquad \sigma,\tau\in T_\rho\Pplus=\one^\perp,
\end{equation}
where $\dagger$ denotes the Moore--Penrose inverse. Thus the \emph{inverse metric operators}, not the metric tensors, are added in \eqref{eq:operator}.

For the relative entropy
\[
 H(\rho)=\sum_i\rho_i\log(N\rho_i),\qquad h_i=\log(N\rho_i),
\]
and a potential $\phi\in\R^N/\R\one$, set
\[
 \A_\lambda(\rho,\phi)=\phi^TK_\lambda\phi,\qquad
 \B_\lambda(\rho,\phi)=\Hess_{g_\lambda}H(\rho)[K_\lambda\phi,K_\lambda\phi].
\]
We investigate the optimal lower bound
\begin{equation}\label{eq:kappa}
 \kappa_\lambda(G)=\inf_{\rho\in\Pplus,\,\phi\notin\R\one}
 \frac{\B_\lambda(\rho,\phi)}{\A_\lambda(\rho,\phi)}.
\end{equation}
The word \emph{entropic curvature} below refers precisely to \eqref{eq:kappa}. It does not denote the sectional curvature or the Ricci tensor of $g_\lambda$.

\subsection{Boundary and global classification}
Our main theorem is local at each boundary support, with a global graph-theoretic consequence.

\begin{theorem}[Two-step criterion]\label{thm:classification}
Fix $\lambda>0$.
\begin{enumerate}
\item Let $\rho^*\in\Pbar$ and $S=\supp\rho^*$. There exist a relative neighborhood $U$ of $\rho^*$ in $\Pbar$ and $M<\infty$ such that
\begin{equation}\label{eq:local-bound}
 \B_\lambda(\rho,\phi)\ge-M\A_\lambda(\rho,\phi)
 \quad(\rho\in U\cap\Pplus,\ \phi\in\R^N)
\end{equation}
if and only if $\max_{v\in V}\dist(v,S)\le2$.
\item Consequently,
\[
 \kappa_\lambda(G)>-\infty
 \quad\Longleftrightarrow\quad \diam(G)\le2.
\]
If $\diam(G)\ge3$, then $\kappa_\lambda(G)=-\infty$ for every $\lambda>0$.
\end{enumerate}
\end{theorem}

The bound in \eqref{eq:local-bound} holds simultaneously for all potentials. In particular, the potential may vary along a sequence of probabilities approaching the boundary. No common scale, power-law relation, or lower bound on the ratios of vanishing coordinates is imposed. The finite constant may depend on $G$, $\lambda$, and the boundary neighborhood.

At $\lambda=0$, finite graphs have a finite pure-transport lower bound without a diameter restriction. We recover this fact directly in \cref{lem:inflow}. The positive-parameter criterion is therefore a genuine singular-perturbation statement.

\begin{theorem}[Stability in the Fisher parameter]\label{thm:lipschitz}
If $\diam(G)\le2$, then $\lambda\mapsto\kappa_\lambda(G)$ is Lipschitz on every finite interval $[0,\Lambda]$. In particular, $\kappa_0(G)>0$ implies $\kappa_\lambda(G)>0$ for all sufficiently small $\lambda>0$.
\end{theorem}

\Cref{thm:lipschitz} includes the pure-transport endpoint; Lipschitz continuity does not require monotonicity of $\kappa_\lambda$.

\subsection{Sharp dissipation and stability}
Let $L$ be the positive weighted Laplacian, $(Lf)_i=\sum_jw_{ij}(f_i-f_j)$. Products such as $LW(\rho)$ below denote ordinary matrix composition (apply $W(\rho)$ first, then $L$). The negative entropy gradient flow is
\begin{equation}\label{eq:flow}
 \dot\rho=-L\rho-\lambda\rho\odot(h-H(\rho)\one),
\end{equation}
where $\odot$ denotes coordinatewise multiplication. For a positive probability, write
\[
 R(\rho)=\osc_i\log\rho_i=\max_i\log\rho_i-\min_i\log\rho_i.
\]
For ordered pairs $i\ne j$ and $r\ge0$, define
\begin{align}
 F_{ij}(r)&=w_{ij}(e^r-e^{-r})
 +\sum_{k\notin\{i,j\}}\min\{w_{ik}(1-e^{-r}),\,w_{jk}(e^r-1)\},\notag\\
 F_G(r)&=\min_{i\ne j}F_{ij}(r).\label{eq:FG}
\end{align}

\begin{theorem}[Optimal instantaneous amplitude dissipation]\label{thm:amplitude}
For every $\lambda\ge0$ and every positive solution of \eqref{eq:flow},
\begin{equation}\label{eq:R-dissipation}
 D^+R(\rho_t)\le-\lambda R(\rho_t)-F_G(R(\rho_t)),
\end{equation}
where $D^+$ is the upper right Dini derivative. For each $r>0$, a positive two-level probability of amplitude $r$ attains equality at time zero. Hence $F_G$ is the largest additional dissipation function valid for all positive initial probabilities when only $G$ and the current amplitude are specified.

Moreover, with
\begin{equation}\label{eq:cG}
 c_G=\min_{i\ne j}\left(2w_{ij}+\sum_{k\notin\{i,j\}}\min\{w_{ik},w_{jk}\}\right),
\end{equation}
one has $F_G(r)/r\to c_G$ as $r\downarrow0$ and $F_G(r)\ge c_G(1-e^{-r})$. In particular, for every fixed $\lambda>0$,
\begin{equation}\label{eq:common-threshold}
 \kappa_\lambda(G)>-\infty
 \quad\Longleftrightarrow\quad\diam(G)\le2
 \quad\Longleftrightarrow\quad c_G>0.
\end{equation}
If $\diam(G)\ge3$, then $F_G\equiv0$.
\end{theorem}

Instantaneous attainability does not imply that a single trajectory is an equality case at every time. Similarly, $F_G\equiv0$ does not mean that the graph heat flow fails to mix; it means that a strictly positive extra instantaneous rate cannot be inferred from the present amplitude alone.

The main contributions are the uniform boundary classification in \cref{thm:classification}, its parameter stability, and the sharp amplitude function in \cref{thm:amplitude}. \Cref{sec:entropy,sec:hilbert} give complementary consequences: an amplitude-dependent variance--entropy inequality, graph-dependent entropy prefactors, and exact contraction bounds in Hilbert's projective metric. The projective-contraction results are stated with their exact uniform and initial-state-dependent regimes.

\subsection{Relation to earlier work}
Maas \cite{Maas2011}, Erbar and Maas \cite{ErbarMaas2012}, and Mielke \cite{Mielke2013} established discrete entropy gradient structures and entropic curvature for finite reversible chains. The Hamiltonian formula for the entropy Hessian used here is part of that theory. Li and Lu \cite{LiLu2022} developed a general framework of probability-dependent edge mobilities; the interior operator \eqref{eq:operator} belongs to that framework. Our problem is the uniform lower bound over all interior probabilities approaching arbitrary boundary supports, rather than the existence of an interior Hessian formula.

For finite reversible chains, Kamtue, Liu, M\"unch, and Peyerimhoff \cite[Theorems~1.3 and~4.2; Lemma~4.1]{KamtueEtAl2025} give explicit universal curvature bounds and a related concavity--Young-inequality estimate. We compare the constants and normalizations after \cref{lem:inflow}; the additional positive inflow retained here is used to control the state-dependent Fisher contribution. For nonlinear Markov dynamics, Erbar, Fathi, and Schlichting \cite{ErbarFathiSchlichting2020} assume Lipschitz dependence of the transition rates on the closed simplex. The effective logarithmic-mean rates of our model fail that boundary assumption. \Cref{prop:representation} makes both the inclusion in the interior mobility framework and the obstruction to a fixed-chain representation explicit.

The mixing of transport and Fisher-type reaction also has a continuous counterpart in spherical Hellinger--Kantorovich geometry. Entropy dissipation without geodesic convexity is developed by Kondratyev and Vorotnikov \cite{KondratyevVorotnikov2019}; further global decay results appear in Mielke and Zhu \cite{MielkeZhu2025}. Carrillo, Chen, Huang, Huang, and Wei \cite{CarrilloEtAl2026} study Fisher--Rao geodesic convexity and functional inequalities. The novelty considered here is the graph-theoretic classification of uniform boundary Hessian bounds for the specified mixed mobility.

The coefficient $c_G$ in \eqref{eq:cG} is the classical continuous-time Dobrushin--Hopf oscillation coefficient; see Gaubert and Qu \cite{GaubertQu2015}. Our amplitude result retains the full dependence on $r$, proves pointwise attainability, and relates its graph threshold to the boundary Hessian criterion. The Hilbert estimate has a mass-conserving Fisher damping term and distinguishes contraction between arbitrary pairs from decay toward the fixed uniform equilibrium.

Recent related work further separates the settings. Mao, Wang, and Xue \cite{MaoWangXue2026} study nonnegative measures on finite Markov chains with a source constrained to a fixed positive spatial profile; the source is controlled by one scalar amplitude. Our state space has fixed total mass, and the Fisher covariance acts in every tangent direction. Crucinio and Pathiraja \cite{CrucinioPathiraja2026} prove preservation of strong log-concavity and convergence estimates for continuous-space Wasserstein--Fisher--Rao flows under conditions on the target distribution. The issue here is instead a uniform entropy-Hessian bound at every boundary support of a finite probability simplex. These results provide context for the two-step boundary criterion and its parameter and amplitude consequences.

\subsection{Proof strategy and scope}
The proof of \cref{thm:classification} begins with the exact expansion $\B_\lambda=\B_0+\lambda\B_1+\lambda^2\B_2$. A nonnegative inflow form is retained in $\B_0$, while all potentially singular negative terms are controlled by one entropy-weighted quadratic form. Two scalar logarithmic-mean estimates allow arbitrary endpoint ratios. They lead to an absorption argument through the support, its first neighborhood, and its second neighborhood. A finite cover of the closed simplex gives the global lower bound. A distance-three test family proves the converse. A transport-only version of the same absorption argument proves \cref{thm:lipschitz}.

The curvature statements concern the interior Riemannian Hessian for a uniform reference probability. Extension to the metric completion or a nonuniform reversible reference law requires separate analysis. The proofs of the classification and stability theorems use analytic estimates uniform in all interior states and potential directions.

\section{Geometry and exact entropy Hessian}\label{sec:geometry}

\subsection{The action and the Fisher normalization}
The kernel of $W(\rho)$ is $\R\one$, by positivity of the edge mobilities and connectedness. Also
\[
 C_\rho=\sum_{i<j}\rho_i\rho_j(e_i-e_j)(e_i-e_j)^T.
\]
Thus, for $\psi=\phi-(\rho^T\phi)\one$,
\begin{equation}\label{eq:action}
 \A_\lambda=\AW+\lambda\AC,\quad
 \AW=\sum_{\{i,j\}\in E}w_{ij}\theta_{ij}(\phi_i-\phi_j)^2,\quad
 \AC=\sum_i\rho_i\psi_i^2=\Var_\rho(\phi),
\end{equation}
where $\theta_{ij}=\theta(\rho_i,\rho_j)$. In particular, $\A_\lambda>0$ for every nonconstant potential.

One may also describe \eqref{eq:metric} by an infimum over transport and reaction. Choose an orientation of the graph, let $B$ be its incidence matrix, and let $M_\rho$ be the diagonal matrix of edge mobilities $w_{ij}\theta_{ij}$. For $\lambda>0$ and $\sigma\in\one^\perp$,
\begin{equation}\label{eq:infconvolution}
 \sigma^TK_\lambda^\dagger\sigma
 =\min_{BJ+r=\sigma,\ \one^Tr=0}
 \left\{J^TM_\rho^{-1}J+\lambda^{-1}\sum_i\frac{r_i^2}{\rho_i}\right\}.
\end{equation}
Indeed, the Lagrange multiplier $\phi$ for the constraint gives
$J=M_\rho B^T\phi$ and $r=\lambda C_\rho\phi$. The constraint becomes $K_\lambda\phi=\sigma$, and the value equals $\phi^TK_\lambda\phi$. Strict convexity in $(J,r)$ proves minimality. Formula \eqref{eq:infconvolution} specifies the mass-conserving reaction normalization.

The integral representation \eqref{eq:logmean} shows that $\theta$ is positive, symmetric, increasing, jointly concave, and homogeneous of degree one. In particular,
\begin{equation}\label{eq:euler}
 a\theta_1(a,b)+b\theta_2(a,b)=\theta(a,b),\qquad
 \theta(a,b)(\log a-\log b)=a-b.
\end{equation}
Since $\nabla H=h+\one$ and $K_\lambda\one=0$, these identities give $Wh=L\rho$ and $C_\rho h=\rho\odot(h-H\one)$, proving \eqref{eq:flow}.

\subsection{Hamiltonian formula and parameter coefficients}
\begin{proposition}[Exact Hessian expansion]\label{prop:hessian}
Let $\sigma=K_\lambda\phi$, and choose a smooth extension of $K_\lambda$ off the affine probability hyperplane. Set $q_i=\phi^T(\partial_iK_\lambda)\phi$. Then
\begin{equation}\label{eq:hamilton-hessian}
 \B_\lambda
 =\sigma^TD_\rho^{-1}\sigma+h^TDK_\lambda[\sigma]\phi
 -\tfrac12(K_\lambda h)^Tq,
 \qquad D_\rho=\diag\rho.
\end{equation}
The expression is independent of the chosen extension. With $\psi=\phi-(\rho^T\phi)\one$, it has the expansion
\begin{equation}\label{eq:quadratic}
 \B_\lambda=\B_0+\lambda\B_1+\lambda^2\B_2,
\end{equation}
where
\begin{align}
 \B_0&=\phi^T(LW)\phi-\frac12\sum_{\{i,j\}\in E}w_{ij}
 [\theta_1(\rho_i,\rho_j)(L\rho)_i+\theta_2(\rho_i,\rho_j)(L\rho)_j]
 (\phi_i-\phi_j)^2,\label{eq:B0}\\
 \B_1&=\frac12\sum_i(L\rho)_i\psi_i^2
 +\sum_{\{i,j\}\in E}w_{ij}
 \left[\frac{\rho_i+\rho_j}{4}+\frac{\theta_{ij}}4(h_i+h_j-2H+6)\right]
 (\phi_i-\phi_j)^2,\label{eq:B1}\\
 \B_2&=\sum_i\rho_i\psi_i^2\left(1+\frac{h_i-H}{2}\right).\label{eq:B2}
\end{align}
\end{proposition}

\begin{proof}
The kinetic Hamiltonian is $\frac12\phi^TK_\lambda(\rho)\phi$. Its constrained equations are
\[
 \dot\rho=K_\lambda\phi,\qquad \dot\phi=-\tfrac12q+c(t)\one.
\]
Consequently $\ddot\rho=DK_\lambda[\sigma]\phi-\frac12K_\lambda q$. Since the Euclidean Hessian of $H$ is $D_\rho^{-1}$, two derivatives of $H$ along this geodesic give \eqref{eq:hamilton-hessian}. Two extensions agreeing on the probability hyperplane have identical tangent derivatives, while their corresponding $q$ vectors differ by a multiple of $\one$. This difference is killed by $K_\lambda$.

For the coefficients, choose the statewise gauge $\rho^T\phi=0$. This choice is made at the evaluation point; it is not differentiated as a varying potential. Put
\[
 F=L\rho,\qquad \ell=h-H\one,\qquad r=\rho\odot\ell.
\]
For $z\in\one^\perp$,
\[
 DC_\rho[z]=D_z-z\rho^T-\rho z^T,\qquad
 Dr[z]_i=(\ell_i+1)z_i-\rho_i h^Tz.
\]
Symmetry of the tangent derivatives gives an equivalent Hessian identity,
\[
 \B_\lambda=\phi^TD(K_\lambda h)[K_\lambda\phi]
 -\tfrac12\phi^TDK_\lambda[K_\lambda h]\phi.
\]
As $K_\lambda h=F+\lambda r$, its coefficients are
\begin{align*}
 \B_0&=\phi^T(LW)\phi-\tfrac12\phi^TDW[F]\phi,\\
 \B_1&=\phi^TLC_\rho\phi+\phi^TDr[W\phi]
       -\tfrac12\phi^TDW[r]\phi-\tfrac12\phi^TDC_\rho[F]\phi,\\
 \B_2&=\phi^TDr[C_\rho\phi]-\tfrac12\phi^TDC_\rho[r]\phi.
\end{align*}
The first identity yields \eqref{eq:B0}. Because $C_\rho\phi=\rho\odot\phi$ in the chosen gauge, the last gives \eqref{eq:B2}.

For completeness, we detail the mixed coefficient. On an edge with $a=\rho_i$, $b=\rho_j$, use
\[
 (\phi_i-\phi_j)(A\phi_i-B\phi_j)
 =\tfrac12(A+B)(\phi_i-\phi_j)^2
  +\tfrac12(A-B)(\phi_i^2-\phi_j^2).
\]
Summing this identity yields
\begin{align*}
 \phi^TLC_\rho\phi-\tfrac12\phi^TDC_\rho[F]\phi
 &=\tfrac12\sum_{\{i,j\}\in E}w_{ij}(a+b)(\phi_i-\phi_j)^2,\\
 \phi^TDr[W\phi]
 &=\tfrac12\sum_{\{i,j\}\in E}w_{ij}\theta_{ij}(\ell_i+\ell_j+2)(\phi_i-\phi_j)^2
   +\tfrac12\sum_iF_i\phi_i^2.
\end{align*}
For the remaining term, differentiating the logarithmic mean in the direction $(a\ell_i,b\ell_j)$ gives
\[
 \theta_1a\ell_i+\theta_2b\ell_j
 =\frac{a+b}{2}+\frac{\theta_{ij}}2(\ell_i+\ell_j-2).
\]
For unequal endpoints this follows by differentiating the quotient in \eqref{eq:logmean}; at equal endpoints it follows by continuity. Combining the three expressions proves \eqref{eq:B1}. Restoring $\psi$ in the vertex terms makes the formulas gauge independent.
\end{proof}

\subsection{Why fixed-chain theory does not settle the boundary problem}
\begin{proposition}[Representation comparison]\label{prop:representation}
For $\lambda>0$, the interior operator has the complete-graph mobility representation
\begin{equation}\label{eq:effective-mobility}
 K_\lambda=\sum_{i<j}[w_{ij}\theta(\rho_i,\rho_j)+\lambda\rho_i\rho_j]
 (e_i-e_j)(e_i-e_j)^T.
\end{equation}
It has no representation on the whole positive simplex by a fixed uniform-reversible chain with logarithmic-mean mobility. If $N\ge3$, it also has no representation by fixed pair coefficients multiplied by degree-one homogeneous pair means on that whole simplex.
\end{proposition}
\begin{proof}
The covariance identity gives \eqref{eq:effective-mobility}. Off-diagonal matrix entries determine the mobility of each pair. A fixed logarithmic-mean representation would therefore require the rate
\[
 w_{ij}+\lambda\frac{\rho_i\rho_j}{\theta(\rho_i,\rho_j)}
\]
to be independent of $\rho$. For $N\ge3$, choose $\rho_i=\rho_j=a$ and vary $a$ while leaving the remaining masses positive; the added rate is $\lambda a$. For $N=2$, the added factor equals $\lambda/2$ at equal masses and tends to zero at an endpoint. Both contradict constancy.

For $N\ge3$, the admissible pair domain $a,b>0$, $a+b<1$ is open. Whenever $(a,b)$ and $(ta,tb)$ both lie in this domain,
\[
 [w_{ij}\theta(ta,tb)+\lambda t^2ab]
 -t[w_{ij}\theta(a,b)+\lambda ab]=\lambda(t^2-t)ab.
\]
This contradicts degree-one homogeneity for $t\ne1$. The same scaling argument does not apply to the two-state slice $a+b=1$; there the extension $w\theta(a,b)+\lambda ab/(a+b)$ is homogeneous.
\end{proof}

Along a simplex boundary path with $a=\varepsilon\downarrow0$ and $b\to b_0>0$, the added rate satisfies
\[
 \lambda\frac{ab}{\theta(a,b)}
 =\lambda\frac{ab\log(b/a)}{b-a}
 \sim\lambda\varepsilon\log(b_0/\varepsilon).
\]
Its ratio to $\varepsilon$ diverges, so its continuous zero extension is not Lipschitz. For $N=2$ one takes $b=1-\varepsilon$; for $N\ge3$ the other coordinates absorb the mass change. This is the boundary regularity obstruction to directly applying the nonlinear-rate hypotheses of \cite{ErbarFathiSchlichting2020}. The mobility itself is smooth and positive in the interior, as required by the general framework of \cite{LiLu2022} with its zero added-force Hamiltonian choice.

\section{Transport inflow and scalar boundary estimates}\label{sec:estimates}

\subsection{A positive form retained from the pure-transport Hessian}
Write
\[
 n_i=\sum_jw_{ij}\rho_j,\qquad d_i=\sum_jw_{ij},\qquad
 d=\max_i d_i,\qquad \omega=\min_{\{i,j\}\in E}w_{ij}>0.
\]
Define the nonnegative quadratic form
\begin{equation}\label{eq:inflow-form}
 \PP(\rho,\phi)=\frac12\sum_{\{i,j\}\in E}w_{ij}
 [\theta_1(\rho_i,\rho_j)n_i+\theta_2(\rho_i,\rho_j)n_j]
 (\phi_i-\phi_j)^2.
\end{equation}
The crucial point is to retain this form, rather than estimate all transport terms in absolute value.

\begin{lemma}[Retained inflow]\label{lem:inflow}
For every $\rho\in\Pplus$ and $\phi\in\R^N$,
\begin{equation}\label{eq:inflow-bound}
 \B_0\ge\tfrac12\PP-C_0\AW,\qquad
 C_0=2d(1+d/\omega)+d/2.
\end{equation}
In particular, $\kappa_0(G)\ge-C_0$ on every finite connected weighted graph.
\end{lemma}
\begin{proof}
Concavity and Euler's identity in \eqref{eq:euler} imply
\[
 \theta(c,a)\le c\theta_1(a,b)+a\theta_2(a,b),\qquad a,b,c>0.
\]
Let $m_i=\sum_kw_{ik}\theta(\rho_i,\rho_k)$. For an edge $\{i,j\}$, put $a=\rho_i$, $b=\rho_j$. Applying the displayed inequality at $(a,b)$, and its exchanged form, gives
\[
 m_i\le\theta_1n_i+d_i a\theta_2,\qquad
 m_j\le\theta_2n_j+d_j b\theta_1.
\]
Since $n_i\ge\omega b$ and $n_j\ge\omega a$,
\begin{equation}\label{eq:mi-bound}
 m_i+m_j\le(1+d/\omega)(\theta_1n_i+\theta_2n_j).
\end{equation}
Now
\[
 |\phi^T(LW)\phi|
 \le\sum_i\sum_{j,k}w_{ij}w_{ik}\theta(\rho_i,\rho_k)
 |\phi_i-\phi_j|\,|\phi_i-\phi_k|.
\]
Use $|xy|\le Mx^2/2+y^2/(2M)$ with $x=\phi_i-\phi_k$ and $y=\phi_i-\phi_j$. The first sum is at most $Md\AW$. The second is
\[
 \frac1{2M}\sum_{\{i,j\}\in E}w_{ij}(m_i+m_j)(\phi_i-\phi_j)^2
 \le\frac{1+d/\omega}{M}\PP
\]
by \eqref{eq:mi-bound}. Thus
\begin{equation}\label{eq:young-inflow}
 |\phi^T(LW)\phi|\le Md\AW+\frac{1+d/\omega}{M}\PP.
\end{equation}
In \eqref{eq:B0}, substitute $(L\rho)_i=d_i\rho_i-n_i$. Positivity of the endpoint derivatives and Euler's identity give
\[
 \B_0\ge\phi^T(LW)\phi+\PP-\tfrac d2\AW.
\]
Taking $M=2(1+d/\omega)$ in \eqref{eq:young-inflow} proves \eqref{eq:inflow-bound}.
\end{proof}

\begin{remark}[Comparison with fixed-chain bounds]\label{rem:fixed-chain}
Let $Q=I-L/d$ be the uniform-reversible Markov matrix associated with the same graph, where $d=\max_i d_i$. Then $Q_{\min}=\omega/d$ on the nonzero off-diagonal transitions, and the pure-transport Onsager operator for $Q$ is $W/d$. The universal finite-dimensional estimate of Kamtue et al.\ \cite[Theorem~1.3]{KamtueEtAl2025} is $\mathrm{CD}_{\rm ent}(-K,\mathcal N)$ with
\[
 K=\frac12+\frac{4}{Q_{\min}}\left(1+\frac4{\mathcal N^2}\right).
\]
Their \cite[Theorem~4.2]{KamtueEtAl2025} improves the infinite-dimensional bound to $K=\frac12+2/Q_{\min}$. Scaling back to our generator gives
\[
 \kappa_0(G)\ge-\left(\frac d2+\frac{2d^2}{\omega}\right).
\]
The coarser constant $C_0=\frac{5d}{2}+\frac{2d^2}{\omega}$ in \cref{lem:inflow} is not offered as an improvement of the fixed-chain lower bound. Its role is to retain the nonnegative form $\PP$, which controls the additional boundary terms when $\lambda>0$. The cited fixed-chain results do not directly bound the Hessian of the state-dependent mixed operator $K_\lambda$.
\end{remark}

\subsection{A common entropy-weighted control of the negative terms}
Let
\[
 \zeta(x)=x(1-\log x)\quad(0<x\le1),\qquad \zeta(0)=0,
\]
and define
\[
 \eta_i=\zeta(\rho_i)+\sum_jw_{ij}\zeta(\rho_j),\qquad
 \Q_\eta(\rho,\phi)=\sum_i\eta_i\phi_i^2.
\]
Unlike $\A_\lambda$ and $\PP$, the form $\Q_\eta$ is not invariant under adding constants to $\phi$. Unless another gauge is expressly used, it is evaluated with $\rho^T\phi=0$. The elementary estimates
\begin{equation}\label{eq:eta-bounds}
 \eta_i\ge\rho_i,\qquad \eta_i\le K_\eta:=1+d,\qquad
 \sum_i\eta_i\le(1+d)(1+\log N)
\end{equation}
follow from $\zeta\le1$ and $-\sum_i\rho_i\log\rho_i\le\log N$.

\begin{lemma}[Entropy-weighted estimates]\label{lem:entropy-weights}
In the gauge $\rho^T\phi=0$, let $c_1=\max\{d,1\}+1/2$ and $C_\lambda=\lambda c_1+\lambda^2/2$. Then
\[
 \B_1\ge-c_1\Q_\eta,\qquad \B_2\ge-\tfrac12\Q_\eta,
\]
and hence
\begin{equation}\label{eq:master-lower}
 \B_\lambda\ge\tfrac12\PP-C_0\AW-C_\lambda\Q_\eta.
\end{equation}
With $D_G=\max\{d,1\}$, one may take
\[
 C=\frac{7+\log N}{2}D_G+2+\frac{\log N}{2}
\]
in the bound
\begin{equation}\label{eq:absolute-controls}
 |\B_1|+|\B_2|+\AC\le C\Q_\eta
\end{equation}
for all positive probabilities and all centered potentials.
\end{lemma}
\begin{proof}
For $1\ge a\ge b>0$, write $a=e^{-u}$ and $b=ae^{-v}$, with $u,v\ge0$. The value at $v=0$ is understood by continuity. Then
\begin{equation}\label{eq:log-weight}
 \theta(a,b)(-\log a-\log b)
 =a(1-e^{-v})(1+2u/v)
 \le a+2au\le2[\zeta(a)+\zeta(b)].
\end{equation}
Because $H\le\log N$, one has
$h_i+h_j-2H\ge\log\rho_i+\log\rho_j$ and $h_i-H\ge\log\rho_i$.
Dropping the positive terms in \eqref{eq:B1}, using $(L\rho)_i\ge-n_i$, and applying \eqref{eq:log-weight} and $(\phi_i-\phi_j)^2\le2\phi_i^2+2\phi_j^2$, we obtain
\[
 \B_1\ge-\sum_i\left[d_i\zeta(\rho_i)+\sum_jw_{ij}\zeta(\rho_j)+\frac{n_i}{2}\right]\phi_i^2
 \ge-c_1\Q_\eta.
\]
Here $n_i\le\sum_jw_{ij}\zeta(\rho_j)\le\eta_i$. Formula \eqref{eq:B2} similarly gives
\[
 \B_2\ge-\frac12\sum_i\rho_i(-\log\rho_i)\phi_i^2\ge-\tfrac12\Q_\eta.
\]
Combining these inequalities with \cref{lem:inflow} proves \eqref{eq:master-lower}.

For \eqref{eq:absolute-controls}, use $|h_i-H|\le-\log\rho_i+\log N$ and $|(L\rho)_i|\le d_i\rho_i+n_i$. If $E_{ij}$ denotes the bracket in \eqref{eq:B1}, then \eqref{eq:log-weight} and $\theta_{ij}\le(\rho_i+\rho_j)/2$ give
\[
 |E_{ij}|\le\left(\frac32+\frac{\log N}{4}\right)
       [\zeta(\rho_i)+\zeta(\rho_j)].
\]
Thus the edge terms contribute at most
$(3+\frac12\log N)D_G\Q_\eta$, and the vertex term contributes at most $D_G\Q_\eta/2$. Hence
\[
 |\B_1|\le\frac{7+\log N}{2}D_G\Q_\eta.
\]
Likewise \eqref{eq:B2} gives
$|\B_2|\le(1+\frac12\log N)\Q_\eta$, and $\AC\le\Q_\eta$.
Adding these estimates proves the stated constant.
\end{proof}

\subsection{Scalar estimates with no ratio restrictions}
The following two lemmas prevent loss of control when different probabilities vanish at unrelated rates.

\begin{lemma}[Derivative balance]\label{lem:scalar-one}
For $0<c\le e^{-2}$ and arbitrary $a,b>0$,
\[
 \theta_1(a,b)+c\theta_2(a,b)\ge\delta(c),\qquad
 \delta(c)=\frac1{4e(1+|\log c|)}.
\]
The coefficient $c$ is independent of the two endpoints.
\end{lemma}
\begin{proof}
If $a\le b$, differentiation under the integral in \eqref{eq:logmean} gives $\theta_1(a,b)\ge1/2$. If $a>b$, set $v=\log(a/b)>0$. By homogeneity,
\[
 \theta_1(a,b)=\frac{v-1+e^{-v}}{v^2}
 =\int_0^1 t e^{-(1-t)v}\dd t,
 \qquad
 \theta_2(a,b)=\frac{e^v-1-v}{v^2}.
\]
For $v\le1$, the integral is at least $1/(2e)$. For $v\ge1$, integrating over $[1-(1+v)^{-1},1]$ gives
\[
 \theta_1(a,b)\ge\frac1{2e(1+v)}.
\]
Put $L=|\log c|\ge2$. If $v\le2L$, this last estimate is at least $\delta(c)$. If $v>2L\ge4$, then $e^v-1-v\ge e^v/2$ and $e^v/v^2$ is increasing. Hence
\[
 c\theta_2(a,b)\ge\frac{ce^v}{2v^2}
 \ge\frac1{8cL^2}\ge\frac18>\delta(c),
\]
where $e^L\ge L^2$ for $L\ge2$ was used.
\end{proof}

\begin{lemma}[Uniform entropy-weight absorption]\label{lem:scalar-two}
For $0<a\le e^{-2}$ and $0<b\le1$, set
\[
 D_*(a,b)=\theta_1(a,b)+a\theta_2(a,b),\qquad L_a=-\log a,
\]
and
\[
 V(a)=\frac{2a}{\delta(a)}+2(1+e^2)a+2aL_a^2.
\]
Then
\begin{equation}\label{eq:scalar-two}
 \frac{\zeta(b)}{D_*(a,b)}\le(1+L_a)V(a)\longrightarrow0
 \quad\text{as }a\downarrow0,
\end{equation}
uniformly for $b\in(0,1]$.
\end{lemma}
\begin{proof}
First we prove $(a+b)/D_*\le V(a)$. If $b\le a$, \cref{lem:scalar-one} gives the first term $2a/\delta(a)$. If $a\le b\le e^2a$, then $\theta_1(a,b)\ge1/2$, giving the second term. In the remaining case, put $v=\log(b/a)\ge2$, so $v\le L_a$. The derivative at the smaller endpoint satisfies
\[
 \theta_1(a,b)=\frac{e^v-1-v}{v^2}\ge\frac{e^v+1}{2v^2},
\]
because $e^v\ge3+2v$ for $v\ge2$. Therefore $(a+b)/D_*\le2av^2\le2aL_a^2$.

If $b\ge a$, then $b(-\log b)/D_*\le L_aV(a)$. If $b\le a\le e^{-2}$, the monotonicity of $x(-\log x)$ on $(0,e^{-1}]$ yields
\[
 \frac{b(-\log b)}{D_*}\le\frac{aL_a}{\delta(a)}\le L_aV(a).
\]
Adding the bound for $b/D_*$ proves \eqref{eq:scalar-two}. Its limit is zero because $V(a)$ is bounded by a constant times $a(1+L_a)^2$.
\end{proof}

\section{Boundary supports and the diameter criterion}\label{sec:boundary}

\subsection{Three-layer absorption}
\begin{lemma}[Local absorption]\label{lem:local-absorption}
Fix $\lambda>0$ and $\rho^*\in\Pbar$. Suppose that $S=\supp\rho^*$ reaches every vertex in at most two steps. For every $\varepsilon>0$, there exist a relative neighborhood $U_\varepsilon$ of $\rho^*$ and a finite constant $C_\varepsilon$ such that
\begin{equation}\label{eq:local-absorption}
 \Q_\eta(\rho,\phi)\le\varepsilon\PP(\rho,\phi)
 +C_\varepsilon\A_\lambda(\rho,\phi)
\end{equation}
for $\rho\in U_\varepsilon\cap\Pplus$ and every potential satisfying $\rho^T\phi=0$.
\end{lemma}
\begin{proof}
Partition the vertices into
\[
 S,\qquad T=\{v:\dist(v,S)=1\},\qquad U=\{v:\dist(v,S)=2\}.
\]
For each $t\in T$, choose an adjacent $s(t)\in S$; for each $u\in U$, choose an adjacent $t(u)\in T$. These choices are fixed independently of $\rho$ and $\phi$. Within each layer, the selected edges are distinct; hence their weighted squares can be summed without counting an edge more than once. Choose $m>0$ and shrink the relative neighborhood of $\rho^*$ so that $\rho_s\ge m$ for all $s\in S$ and, whenever the relevant layer is nonempty,
\[
 \max_{t\in T}\rho_t\le\min\{e^{-2},\,me^{-2}\},
 \qquad \max_{u\in U}n_u\le e^{-2}.
\]
These conditions are possible because $\rho_t^*=0$ for $t\in T$, and no neighbor of $u\in U$ lies in $S$, so $n_u\to0$ as $\rho\to\rho^*$. Empty layers impose no condition. We may further shrink the neighborhood as needed in the uniform estimates below. Denote the coefficient of $(\phi_i-\phi_j)^2$ in \eqref{eq:inflow-form} by $p_{ij}$.

\emph{The support layer.} By \eqref{eq:action},
\begin{equation}\label{eq:support-layer}
 \sum_{s\in S}\phi_s^2\le\frac{\A_\lambda}{\lambda m}.
\end{equation}

\emph{The first layer.} For $t\in T$, $n_t\ge\omega m$. At a small endpoint $\rho_t$, with $\rho_{s(t)}\ge m$, the logarithmic-mean derivative formula used in \cref{lem:scalar-one} gives, uniformly on the chosen neighborhood,
\[
 \theta_2(\rho_{s(t)},\rho_t)\ge\frac{m}{2\rho_t|\log\rho_t|^2},\qquad
 p_{s(t),t}\ge\frac{\omega^2m^2}{4\rho_t|\log\rho_t|^2}\longrightarrow\infty.
\]
Indeed $v=\log(\rho_{s(t)}/\rho_t)\ge2$, $v\le|\log\rho_t|$, and $e^v-1-v\ge e^v/2$. Set
$\gamma=\max_{t\in T}p_{s(t),t}^{-1}$, with value zero for an empty layer. Then $\gamma\to0$, and
\begin{equation}\label{eq:first-layer}
 \sum_{t\in T}\phi_t^2
 \le2\gamma\PP+2|T|\sum_{s\in S}\phi_s^2.
\end{equation}

\emph{The second layer.} Every neighbor of $u\in U$ lies outside $S$, so $n_u\to0$. Fix $t=t(u)$. From $n_t\ge\omega m$ and $n_u\ge\omega\rho_t$, there are constants $c,c'>0$, independent of the varying state, such that
\begin{equation}\label{eq:second-coefficient}
 p_{tu}\ge c(\theta_1(\rho_t,\rho_u)+n_u\theta_2(\rho_t,\rho_u)),\qquad
 p_{tu}\ge c'D_*(\rho_t,\rho_u).
\end{equation}
One may take $c=(\omega/2)\min\{\omega m,1\}$ and $c'=(\omega/2)\min\{\omega m,\omega\}$.

All neighbors of $u$, not merely the selected predecessor, contribute to $\eta_u$. Put $a_j=w_{uj}\rho_j$ for $j\sim u$, so $\sum_ja_j=n_u$. The entropy of $(a_j/n_u)_{j\sim u}$ is at most $\log N$. Therefore, for small $n_u$,
\begin{align*}
 \sum_{j\sim u}w_{uj}\zeta(\rho_j)
 &=n_u+\sum_{j\sim u}a_j(-\log a_j)+\sum_{j\sim u}a_j\log w_{uj}\\
 &\le n_u[1+\log N+\log^+(\max_Ew)-\log n_u]
 \le C_{\mathrm{nbr}} n_u(1-\log n_u).
\end{align*}
Here $C_{\mathrm{nbr}}$ is a finite constant depending only on the graph. By the first inequality in \eqref{eq:second-coefficient} and \cref{lem:scalar-one},
\[
 \frac{n_u(1-\log n_u)}{p_{tu}}
 \le\frac{4e}{c}n_u(1+|\log n_u|)^2\longrightarrow0.
\]
By the second inequality and \cref{lem:scalar-two},
\[
 \frac{\zeta(\rho_u)}{p_{tu}}
 \le\frac{(1+|\log\rho_t|)V(\rho_t)}{c'}\longrightarrow0.
\]
Both limits are uniform in all other vanishing coordinates. Thus
\[
 \beta=\max_{u\in U}\frac{\eta_u}{p_{t(u),u}}\longrightarrow0,
\]
with the usual empty-layer convention. Using \eqref{eq:eta-bounds},
\begin{equation}\label{eq:second-layer}
 \sum_{u\in U}\eta_u\phi_u^2
 \le2\beta\PP+2K_\eta|U|\sum_{t\in T}\phi_t^2.
\end{equation}

Combining \eqref{eq:support-layer}, \eqref{eq:first-layer}, and \eqref{eq:second-layer} yields
\begin{align*}
 \Q_\eta&\le[2\beta+2K_\eta(1+2|U|)\gamma]\PP+C_A\A_\lambda,\\
 C_A&=\frac{K_\eta[1+2(1+2|U|)|T|]}{\lambda m}.
\end{align*}
The coefficient of $\PP$ tends to zero. Shrinking the neighborhood proves \eqref{eq:local-absorption}. If $S=V$, \eqref{eq:support-layer} proves the claim directly.
\end{proof}

\subsection{Sufficiency and the compactness argument}
\begin{proof}[Proof of the sufficient directions in \cref{thm:classification}]
For a fixed $\lambda>0$, apply \cref{lem:local-absorption} with $\varepsilon=1/(4C_\lambda)$. The master estimate \eqref{eq:master-lower} gives
\[
 \B_\lambda\ge\tfrac14\PP-C_0\AW-C_\lambda C_\varepsilon\A_\lambda
 \ge-(C_0+C_\lambda C_\varepsilon)\A_\lambda.
\]
This is \eqref{eq:local-bound}.

If $\diam(G)\le2$, every nonempty support satisfies the two-step condition. The corresponding relative neighborhoods cover the compact closed simplex $\Pbar$. A finite subcover and the maximum of the finitely many constants give a global lower bound. The estimates are quadratic-form inequalities valid for all potentials before the cover is taken. Thus compactness is applied only to the state variable and introduces no fixed-potential restriction.
\end{proof}

\subsection{The distance-three obstruction}
\begin{lemma}[An explicit degenerating family]\label{lem:distance-three}
Suppose $\dist(u,v)\ge3$. For $0<\varepsilon<1/N$, set
\[
 a=1-(N-1)\varepsilon,\qquad \rho_u=a,\quad\rho_i=\varepsilon\ (i\ne u),
 \qquad\phi=e_v.
\]
Let $d_v=\sum_jw_{vj}$ and $s_2(v)=\sum_jw_{vj}^2$. Then
\begin{align}
 \A_\lambda&=\varepsilon[d_v+\lambda(1-\varepsilon)],\notag\\
 \B_\lambda&=\varepsilon\left[d_v^2+s_2(v)+2\lambda d_v+\lambda^2(1-\varepsilon)
 +\frac{\lambda a}{2}\log(\varepsilon/a)[d_v+\lambda(1-2\varepsilon)]\right].\label{eq:explicit-family}
\end{align}
For every fixed $\lambda>0$,
\begin{equation}\label{eq:negative-infinity}
 \frac{\B_\lambda}{\A_\lambda}
 =\frac\lambda2\log\varepsilon+(d_v+\lambda)
 +\frac{s_2(v)}{d_v+\lambda}+o(1)\longrightarrow-\infty.
\end{equation}
\end{lemma}
\begin{proof}
The vertex $v$ and all its neighbors have mass $\varepsilon$, and none of those vertices is adjacent to $u$. Hence $L\rho$ vanishes on the closed neighborhood of $v$. In addition,
\[
 h_v-H=a\log(\varepsilon/a).
\]
Using \cref{prop:hessian}, with the centered version of $e_v$ in the reaction terms, gives
\begin{align*}
 \B_0&=\varepsilon[d_v^2+s_2(v)],\\
 \B_1&=\varepsilon d_v[2+(h_v-H)/2],\\
 \B_2&=\varepsilon(1-\varepsilon)
       +\tfrac12\varepsilon(1-2\varepsilon)(h_v-H).
\end{align*}
Together with \eqref{eq:action}, these identities give \eqref{eq:explicit-family}. Dividing and using $a\to1$ and $\varepsilon|\log\varepsilon|\to0$ proves \eqref{eq:negative-infinity}.
\end{proof}

\begin{proof}[Proof of the necessary directions in \cref{thm:classification}]
The global necessity follows from \cref{lem:distance-three}. For the local assertion, suppose that a vertex $v$ has distance at least three from $S=\supp\rho^*$. Let $m_0=|V\setminus S|$, and set
\[
 \rho_s=(1-m_0\varepsilon)\rho_s^*\ (s\in S),\qquad
 \rho_i=\varepsilon\ (i\notin S),\qquad\phi=e_v.
\]
For small $\varepsilon>0$ this is a positive probability converging to $\rho^*$. Again $L\rho$ vanishes on the closed neighborhood of $v$, and the densities at $v$ and all its neighbors equal $\varepsilon$. The masses on $S$ need not be equal. To verify the coefficient identities directly, set $\psi=e_v-\varepsilon\one$. Since $\sum_i(L\rho)_i=0$, $(L\rho)_v=0$, and $\sum_i\rho_i(h_i-H)=0$, we have
\[
 \sum_i(L\rho)_i\psi_i^2=(1-2\varepsilon)(L\rho)_v=0,
 \qquad
 \sum_i\rho_i\psi_i^2(h_i-H)=\varepsilon(1-2\varepsilon)(h_v-H).
\]
The only edges with a nonzero potential difference are incident to $v$, where both endpoint masses equal $\varepsilon$. Consequently, the same formulas for $\B_0,\B_1,\B_2$ hold with
\[
 h_v-H=\log\varepsilon-\sum_i\rho_i\log\rho_i=\log\varepsilon+O(1).
\]
Therefore $\B_\lambda/\A_\lambda=(\lambda/2)\log\varepsilon+O(1)\to-\infty$. No neighborhood can satisfy \eqref{eq:local-bound}.
\end{proof}

\begin{corollary}[Singular change at zero reaction]\label{cor:singular}
If $\diam(G)\ge3$, then $\kappa_0(G)>-\infty$, whereas $\kappa_\lambda(G)=-\infty$ for every $\lambda>0$. In particular, every weighted hypercube $Q_n$ with $n\ge3$ has this property.
\end{corollary}
\begin{proof}
Apply \cref{lem:inflow,thm:classification}; the graph diameter of $Q_n$ is $n$.
\end{proof}

\section{Parameter stability, including the zero endpoint}\label{sec:parameter}

A direct use of \cref{lem:local-absorption} would introduce $1/\lambda$ through \eqref{eq:support-layer}. To include $\lambda=0$, the support layer must instead be anchored by transport paths.

\begin{lemma}[Transport-only absorption]\label{lem:transport-absorption}
Assume $\diam(G)\le2$. For every $\varepsilon>0$, there is a finite constant $C_\varepsilon$ such that
\begin{equation}\label{eq:transport-absorption}
 \Q_\eta(\rho,\phi)\le\varepsilon\PP(\rho,\phi)+C_\varepsilon\AW(\rho,\phi)
\end{equation}
for all $\rho\in\Pplus$ and every $\phi$ with $\rho^T\phi=0$.
\end{lemma}
\begin{proof}
Work near a fixed $\rho^*\in\Pbar$, and retain the layers and selected edges of \cref{lem:local-absorption}. Choose a root $r\in S$ and use the temporary gauge $\chi=\phi-\phi_r\one$, so $\chi_r=0$. Each $s\in S$ is connected to $r$ by a fixed path of length at most two.

If the whole path lies in $S$, every edge on it has mobility at least $\omega m$. Cauchy--Schwarz along the path bounds $\chi_s^2$ by $2\AW/(\omega m)$. Otherwise the intermediate vertex belongs to $T$, and the inflow coefficients of both path edges tend uniformly to infinity. The same path estimate bounds $\chi_s^2$ by a quantity $o(1)\PP$. Summing over the finitely many chosen paths gives
\[
 \sum_{s\in S}\chi_s^2\le C_S\AW+\alpha(\rho)\PP,
 \qquad\alpha(\rho)\longrightarrow0.
\]
The estimates \eqref{eq:first-layer} and \eqref{eq:second-layer} use only differences and endpoint squares, so they remain valid in this gauge. Hence
\[
 \sum_i\eta_i\chi_i^2\le C\AW+o(1)\PP.
\]
To return to the centered gauge, put $\mu=\rho^T\chi$. Since $\rho^T\phi=0$, one has $\phi=\chi-\mu\one$. Jensen's inequality and \eqref{eq:eta-bounds} imply
\[
 \mu^2\le\sum_i\rho_i\chi_i^2\le\sum_i\eta_i\chi_i^2,
\]
and therefore
\[
 \Q_\eta(\rho,\phi)
 \le2[1+(1+d)(1+\log N)]\sum_i\eta_i\chi_i^2.
\]
Shrink the neighborhood until the coefficient of $\PP$ is at most the prescribed $\varepsilon$. If $\rho^*$ is an interior point, all paths remain in $S=V$ and there is no vanishing coefficient to consider. A finite relative cover of $\Pbar$ gives \eqref{eq:transport-absorption} globally.
\end{proof}

\begin{proof}[Proof of \cref{thm:lipschitz}]
The case $\Lambda=0$ is immediate. Suppose $\Lambda>0$. Choose $\varepsilon\le1/(4C_\Lambda)$ in \cref{lem:transport-absorption} and use \eqref{eq:master-lower}. There is a constant $D<\infty$ such that
\begin{equation}\label{eq:common-coercivity}
 \B_\lambda\ge\tfrac14\PP-D\AW\qquad(0\le\lambda\le\Lambda).
\end{equation}
It follows that $\kappa_\lambda\ge-D$, taking $D\ge0$. An explicit common upper bound is available at the uniform state. Let $\gamma_G>0$ be the smallest positive eigenvalue of $L$, and let $\phi\perp\one$ be a corresponding eigenvector. At $\rho=\one/N$ one has $h=0$ and
\[
 K_\lambda=\frac1N\left(L+\lambda\left(I-\frac{\one\one^T}{N}\right)\right),\qquad
 \B_\lambda=N\|K_\lambda\phi\|^2.
\]
Therefore $\B_\lambda/\A_\lambda=\gamma_G+\lambda$, and
$\kappa_\lambda\le U:=\gamma_G+\Lambda$ on $[0,\Lambda]$.

Fix $\lambda\in[0,\Lambda]$ and a test pair satisfying
\[
 q_\lambda:=\frac{\B_\lambda}{\A_\lambda}\le\kappa_\lambda+1.
\]
Using \eqref{eq:common-coercivity} and $\AW\le\A_\lambda$, we get
\[
 \frac{\PP}{\A_\lambda}\le4(U+1+D).
\]
Apply \cref{lem:transport-absorption} once more, with any fixed positive choice of $\varepsilon$. It follows that $\Q_\eta/\A_\lambda$ and $\AC/\A_\lambda$ are uniformly bounded on these near-minimizing pairs. Also $|q_\lambda|$ is uniformly bounded by the preceding bounds on $\kappa_\lambda$.

For the same pair and $\mu\in[0,\Lambda]$, \eqref{eq:quadratic}, \eqref{eq:action}, and \eqref{eq:absolute-controls} give
\[
 |\B_\mu-\B_\lambda|\le C_\Lambda'|\mu-\lambda|\Q_\eta,
 \qquad |\A_\mu-\A_\lambda|=|\mu-\lambda|\AC.
\]
For $|\mu-\lambda|$ less than a fixed positive number, independent of the near-minimizing pair, $\A_\mu\ge\A_\lambda/2$. The exact quotient identity is
\[
 q_\mu-q_\lambda
 =\frac{(\B_\mu-\B_\lambda)-q_\lambda(\A_\mu-\A_\lambda)}{\A_\mu}.
\]
It yields $|q_\mu-q_\lambda|\le L|\mu-\lambda|$ with a uniform $L$. Take a sequence of pairs for which $q_\lambda\downarrow\kappa_\lambda$, and use $\kappa_\mu\le q_\mu$. This proves
$\kappa_\mu-\kappa_\lambda\le L|\mu-\lambda|$ for small separations. Exchanging the parameters gives the reverse inequality. Subdividing any interval in $[0,\Lambda]$ into finitely many sufficiently short subintervals proves the claimed Lipschitz bound. Its consequence for strictly positive $\kappa_0$ follows by continuity.
\end{proof}

\begin{remark}
The preceding proof needs a relative bound on $\AC/\A_\lambda$ only for near-minimizing pairs. It does not assume a uniform comparison of $\AC$ and $\AW$ for arbitrary states and arbitrary potentials. Likewise, \cref{thm:classification} establishes boundedness from below, not nonnegativity.
\end{remark}

\section{Sharp logarithmic-amplitude dissipation}\label{sec:amplitude}

\subsection{Global positive solutions and entropy dissipation}
\begin{proposition}[Well-posedness and the dissipation identity]\label{prop:flow-basic}
For every $\lambda\ge0$ and $\rho_0\in\Pplus$, \eqref{eq:flow} has a unique global positive solution. Along this solution, the minimum coordinate is nondecreasing and the maximum coordinate is nonincreasing. Moreover,
\begin{equation}\label{eq:entropy-identity}
 -\frac{\dd}{\dd t}H(\rho_t)=\IW(\rho_t)+\lambda\IF(\rho_t),
\end{equation}
where
\begin{align*}
 \IW(\rho)&=\sum_{\{i,j\}\in E}w_{ij}(\rho_i-\rho_j)(\log\rho_i-\log\rho_j),\\
 \IF(\rho)&=\Var_\rho(\log(N\rho)).
\end{align*}
The uniform probability is the only stationary point in $\Pplus$.
\end{proposition}
\begin{proof}
The vector field is smooth on $\Pplus$ and has zero coordinate sum. At a minimal coordinate, the heat contribution $\sum_jw_{ij}(\rho_j-\rho_i)$ is nonnegative, and $h_i\le H$ makes the reaction contribution nonnegative as well. At a maximal coordinate both contributions are nonpositive. The corresponding one-sided derivative statements at tied extrema imply
\[
 0<\min_i\rho_i(0)\le\rho_i(t)\le\max_i\rho_i(0)<1.
\]
Thus the solution remains in a compact subset of $\Pplus$, so standard local existence and uniqueness extend it for all times.

Differentiating $H$ gives $\dot H=-h^TK_\lambda h$. The identities in \eqref{eq:euler} and the covariance formula reduce this to \eqref{eq:entropy-identity}. Both dissipations are nonnegative. If the vector field vanishes, then $\IW=0$; connectedness forces all coordinates to be equal. Conversely the uniform probability is stationary.
\end{proof}

\subsection{The exact instantaneous optimization}
\begin{proof}[Proof of \cref{thm:amplitude}]
For a positive solution, $R(t)=\max_{i,j}(\log\rho_i(t)-\log\rho_j(t))$. The upper right derivative of a finite maximum of differentiable functions is the maximum of the derivatives over its active indices. Fix an active ordered pair $(i,j)$, with $i$ a maximum-density vertex and $j$ a minimum-density vertex. Put
\[
 r=R(t),\quad E=e^r,\quad y_k=\rho_k/\rho_j\in[1,E].
\]
The reaction contribution to the difference of logarithmic derivatives is exactly $-\lambda r$. The negative of the heat contribution is
\begin{equation}\label{eq:affine-loss}
 L_{ij}=w_{ij}(E-E^{-1})
 +\sum_{k\notin\{i,j\}}\{w_{ik}(1-y_k/E)+w_{jk}(y_k-1)\}.
\end{equation}
Each summand is affine in $y_k$, and its minimum on $[1,E]$ is
$\min\{w_{ik}(1-E^{-1}),w_{jk}(E-1)\}$. Consequently $L_{ij}\ge F_{ij}(r)\ge F_G(r)$ for every active pair, proving \eqref{eq:R-dissipation}.

To prove sharpness, choose an ordered pair attaining the global minimum $F_G(r)$. Set $y_i=E$, $y_j=1$, and choose each remaining $y_k\in\{1,E\}$ to minimize its affine summand in \eqref{eq:affine-loss}. Normalize by $\rho_k=y_k/\sum_\ell y_\ell$. The chosen pair is active and its derivative difference is $-\lambda r-F_G(r)$. The full upper right derivative is at least this difference, while the universal bound already proved is at most the same number. Equality follows, including when the extrema are tied. Any strictly larger proposed dissipation value at $r$ would fail at this initial state.

Expanding the finite family $F_{ij}$ at zero gives \eqref{eq:cG}. Since $e^r-1\ge1-e^{-r}$ and $e^r-e^{-r}\ge2(1-e^{-r})$, each $F_{ij}$ is at least its corresponding coefficient in \eqref{eq:cG} times $1-e^{-r}$. This proves $F_G(r)\ge c_G(1-e^{-r})$.

For a given pair, that coefficient is positive exactly when the pair is adjacent or has a common neighbor. Its positivity for every pair is equivalent to $\diam(G)\le2$. If a pair has distance at least three, then its entire function $F_{ij}$ is zero, and therefore $F_G\equiv0$. Combining with \cref{thm:classification} proves \eqref{eq:common-threshold}.
\end{proof}

\subsection{Scalar envelopes}
\begin{corollary}[Amplitude bounds]\label{cor:amplitude-envelopes}
Let $R_0=R(\rho_0)$. The scalar initial-value problem
\[
 \dot r_G=-\lambda r_G-F_G(r_G),\qquad r_G(0)=R_0
\]
has a unique global nonnegative solution, and $R(\rho_t)\le r_G(t)$. In addition,
\begin{equation}\label{eq:amplitude-envelope}
 R(\rho_t)\le\min\left\{R_0e^{-\lambda t},\,
 \log\big[1+(e^{R_0}-1)e^{-(\lambda+c_G)t}\big]\right\}.
\end{equation}
\end{corollary}
\begin{proof}
The function $F_G$ is continuous, locally Lipschitz, nonnegative, and nondecreasing, being a finite minimum of functions with these properties. It vanishes at zero. Standard scalar comparison with \eqref{eq:R-dissipation} yields the asserted envelope. Dropping $F_G$ gives the first bound in \eqref{eq:amplitude-envelope}.

For the second, put $U(t)=e^{R(t)}-1$. By \cref{thm:amplitude},
\[
 D^+U\le-\lambda e^R R-c_G(e^R-1)\le-(\lambda+c_G)U,
\]
where $e^RR\ge e^R-1$ for $R\ge0$. Integration proves the result.
\end{proof}

\subsection{Explicit graph families}
\begin{proposition}[Graph-family formulas]\label{prop:families}
For common edge weight $w>0$, the functions and coefficients in \eqref{eq:FG} and \eqref{eq:cG} are given in \cref{tab:families}. For the four-cycle with coordinate weights $m\le M$,
\begin{equation}\label{eq:anisotropic}
 F_{Q_2}(r)=(1-e^{-r})\min\{m(e^r+1),m+M\},\qquad c_{Q_2}=2m.
\end{equation}
\end{proposition}
\begin{table}[htbp]
\centering
\caption{Sharp amplitude dissipation for selected graph families.}\label{tab:families}
\renewcommand{\arraystretch}{1.3}
\begin{tabular}{@{}lll@{}}
\toprule
Graph & $F_G(r)$ & $c_G$\\
\midrule
$K_N$, $N\ge2$ & $w(1-e^{-r})(e^r+N-1)$ & $Nw$\\
$K_{1,n}$, $n\ge2$ & $w(1-e^{-r})$ & $w$\\
$K_{m,n}$, $m,n\ge2$ & $w(1-e^{-r})\min\{e^r+1,m,n\}$ & $2w$\\
$Q_2=K_{2,2}$ & $2w(1-e^{-r})$ & $2w$\\
$\diam(G)\ge3$ & $0$ & $0$\\
\bottomrule
\end{tabular}
\end{table}
\begin{proof}
For $K_N$, every other vertex contributes $w(1-e^{-r})$, in addition to the direct-edge term $w(e^r-e^{-r})$. For a star, a leaf-to-leaf pair contributes $w(1-e^{-r})$; each center-to-leaf orientation has at least the larger direct-edge loss.

For $K_{m,n}$, two vertices on the same side have loss equal to the size of the opposite side times $w(1-e^{-r})$. An adjacent pair on opposite sides has no common neighbor, and its loss is $w(e^r-e^{-r})$. Taking the minimum of these three pair types gives the displayed formula. In particular,
\[
 c_{K_{m,n}}=w\min\{2,m,n\}=2w\qquad(m,n\ge2).
\]
The adjacent-pair term must be included; taking only the same-side pairs would incorrectly give $w\min\{m,n\}$ when both sides have more than two vertices. The four-cycle is the case $m=n=2$.

For the anisotropic four-cycle, put $E=e^r$. Adjacent pairs contribute $m(E-E^{-1})$ or $M(E-E^{-1})$. For an opposite pair the two common neighbors have interchanged incident weights, so its contribution is
\begin{align*}
 &(1-E^{-1})\bigl[\min\{m,ME\}+\min\{M,mE\}\bigr]\\
 &\hspace{2em}=(1-E^{-1})\bigl[m+\min\{M,mE\}\bigr].
\end{align*}
For $E\le M/m$ this equals $m(E-E^{-1})$; for $E\ge M/m$ it equals $(m+M)(1-E^{-1})$. Taking the minimum over all ordered pairs gives \eqref{eq:anisotropic}, including $m=M$. The last row follows from \cref{thm:amplitude}.
\end{proof}

\begin{remark}
The star formula shows why the small-amplitude coefficient cannot be extended to the bound $F_G(r)\ge c_Gr$ for all $r>0$. Indeed $w(1-e^{-r})<wr$. For the two-state graph $K_{1,1}$, one must use the complete-graph formula, which gives $F_G(r)=2w\sinh r$.
\end{remark}

\section{Entropy estimates without a curvature assumption}\label{sec:entropy}

\subsection{A sharp amplitude-dependent variance inequality}
Define
\[
 a(r)=\frac{r^2}{e^r-1-r}\quad(r>0),\qquad a(0)=2.
\]
The reciprocal has the useful representation
\[
 \frac1{a(r)}=\int_0^1 s e^{(1-s)r}\dd s.
\]
It follows that $a$ is decreasing and $0<a(r)\le2$.

\begin{theorem}[Variance--entropy comparison]\label{thm:variance}
Every $\rho\in\Pplus$ satisfies
\begin{equation}\label{eq:variance-entropy}
 \IF(\rho)\ge a(R(\rho))H(\rho).
\end{equation}
For each fixed $r>0$, the constant is optimal when the number of states may vary:
\[
 \inf_{N\ge2}\ \inf_{\rho\in\Pplus,\ R(\rho)=r}
 \frac{\Var_\rho(\log(N\rho))}{H(\rho)}=a(r).
\]
This does not assert optimality for each fixed $N$.
\end{theorem}
\begin{proof}
Let $h_i=\log(N\rho_i)$ and $\pi_i=1/N$. Define the exponential tilt
\[
 \Psi(s)=\log\sum_i\pi_i e^{s h_i},\qquad
 \mu_s(i)=\pi_i e^{s h_i-\Psi(s)},\qquad0\le s\le1.
\]
Then $\Psi(0)=\Psi(1)=0$, $\Psi'(1)=H$, and $\Psi''(s)=\Var_{\mu_s}(h)$. Integration by parts gives
\begin{equation}\label{eq:tilt-identity}
 H=\int_0^1s\Var_{\mu_s}(h)\dd s.
\end{equation}
Writing $R=\osc h$, we have
\[
 \frac{\mu_s(i)}{\rho_i}
 =\frac{e^{-(1-s)h_i}}{\sum_j\rho_j e^{-(1-s)h_j}}
 \le e^{(1-s)R}.
\]
Evaluate the variance under $\mu_s$ at the mean of $h$ under $\rho$ to obtain
\[
 \Var_{\mu_s}(h)\le e^{(1-s)R}\Var_\rho(h).
\]
Together with \eqref{eq:tilt-identity}, this proves \eqref{eq:variance-entropy}. The uniform probability, for which both sides vanish, is included by continuity.

For sharpness, fix $r>0$, put $E=e^{-r}$ and $Z=N-1+E$, and take $N-1$ masses equal to $1/Z$ and one mass equal to $E/Z$. These probabilities have amplitude exactly $r$, and direct calculation gives
\[
 H_N=\log(N/Z)-rE/Z,\qquad
 \mathscr I_{F,N}=\frac{r^2(N-1)E}{Z^2}.
\]
As $N\to\infty$,
\[
 NH_N\longrightarrow1-(1+r)e^{-r},\qquad
 N\mathscr I_{F,N}\longrightarrow r^2e^{-r}.
\]
Their ratio tends to $a(r)$, proving the stated infimum.
\end{proof}

\subsection{A graph-dependent entropy prefactor}
Suppose a pure-transport entropy-production bound is available:
\begin{equation}\label{eq:transport-entropy}
 \IW(\rho)\ge\alpha_GH(\rho)\quad\text{for every }\rho\in\Pplus,
 \qquad\alpha_G\ge0.
\end{equation}
This is a separate functional-inequality input, not a nonnegative-curvature assumption for the mixed metric. Taking $\alpha_G=0$ is always valid.

For $\lambda>0$, define
\begin{equation}\label{eq:graph-prefactor}
 \mathcal C_{G,\lambda}(R)
 =\exp\left\{\lambda\int_0^R\frac{2-a(s)}{\lambda s+F_G(s)}\dd s\right\},
 \qquad \mathcal C_{G,\lambda}(0)=1.
\end{equation}
The dependence on $\lambda$ is displayed explicitly. Since $2-a(s)=2s/3+O(s^2)$, the integrand has a removable singularity at zero.

\begin{theorem}[Entropy decay with an improved prefactor]\label{thm:entropy-decay}
Assume \eqref{eq:transport-entropy}. For $\lambda>0$ and every positive solution,
\begin{equation}\label{eq:entropy-decay}
 H(\rho_t)\le\mathcal C_{G,\lambda}(R_0)H(\rho_0)
 e^{-(\alpha_G+2\lambda)t},\qquad R_0=R(\rho_0).
\end{equation}
Set
\begin{equation}\label{eq:baseline-prefactor}
 \mathcal C(R)=\frac{R^2}{2[1-(1+R)e^{-R}]}\quad(R>0),\qquad\mathcal C(0)=1.
\end{equation}
Then $\mathcal C_{G,\lambda}(R)\le\mathcal C(R)$. For $R>0$, the inequality is strict if $\diam(G)\le2$, while equality holds if $\diam(G)\ge3$. An explicit upper bound without integration is
\begin{equation}\label{eq:prefactor-power}
 \mathcal C_{G,\lambda}(R)
 \le\mathcal C(R)^{\lambda/[\lambda+c_G k(R)]},\qquad
 k(R)=\frac{1-e^{-R}}R,\quad k(0)=1.
\end{equation}
No optimality of these entropy prefactors for a fixed graph is asserted.
\end{theorem}
\begin{proof}
The equilibrium case is immediate. Otherwise use \eqref{eq:entropy-identity}, \eqref{eq:transport-entropy}, and \cref{thm:variance}. With the scalar envelope $r_G$ from \cref{cor:amplitude-envelopes} and the monotonicity of $a$,
\[
 \frac{\dd}{\dd t}\log H(\rho_t)
 \le-\alpha_G-\lambda a(r_G(t))
 =-(\alpha_G+2\lambda)+\lambda[2-a(r_G(t))].
\]
Since $\lambda>0$, $r_G$ decreases to zero. Integrate in time and bound the correction by its integral over $[0,\infty)$. The substitution
$\dot r_G=-\lambda r_G-F_G(r_G)$ gives exactly the logarithm of \eqref{eq:graph-prefactor}, proving \eqref{eq:entropy-decay}.

Direct differentiation of \eqref{eq:baseline-prefactor}, followed by its limit at zero, gives
\[
 \log\mathcal C(R)=\int_0^R\frac{2-a(s)}s\dd s.
\]
Comparing denominators proves $\mathcal C_{G,\lambda}\le\mathcal C$. If the diameter is at most two, $F_G(s)>0$ for every $s>0$, and $2-a(s)>0$, so the inequality is strict for $R>0$. If the diameter is at least three, $F_G\equiv0$, giving equality.

Finally, $k$ is decreasing, and \cref{thm:amplitude} gives $F_G(s)\ge c_G k(R)s$ for $0\le s\le R$. Substitution in \eqref{eq:graph-prefactor} proves \eqref{eq:prefactor-power}.
\end{proof}

At $\lambda=0$, one uses \eqref{eq:entropy-identity} and \eqref{eq:transport-entropy} directly to obtain $H(\rho_t)\le H(\rho_0)e^{-\alpha_Gt}$; \eqref{eq:graph-prefactor} is not defined by multiplying an undefined integral by zero. When $\lambda>0$, even the choice $\alpha_G=0$ provides exponential entropy decay on every connected graph, including graphs whose entropic curvature is $-\infty$ by \cref{cor:singular}.

\begin{corollary}[An amplitude-only entropy bound]\label{cor:amplitude-entropy}
For every $\lambda\ge0$ and every positive solution,
\[
 H(\rho_t)\le\frac18\log^2\big[1+(e^{R_0}-1)e^{-(\lambda+c_G)t}\big]
 \le\frac{(e^{R_0}-1)^2}{8}e^{-2(\lambda+c_G)t}.
\]
\end{corollary}
\begin{proof}
A random variable with range of length $R$ has variance at most $R^2/4$: evaluate its second moment about the midpoint of that range. Apply this to $h$ under every $\mu_s$ in \eqref{eq:tilt-identity}. It gives $H\le R^2/8$. Now use \eqref{eq:amplitude-envelope} and $\log(1+x)\le x$.
\end{proof}

The initial amplitude can depend on $N$. Thus neither \cref{thm:entropy-decay} nor \cref{cor:amplitude-entropy} implies a dimension-independent absolute mixing time for arbitrary initial probabilities. The sharpness in \cref{thm:variance} permits changing $N$, whereas the instantaneous sharpness in \cref{thm:amplitude} keeps the graph fixed. These are distinct optimization statements.

\section{Projective contraction between two solutions}\label{sec:hilbert}

\subsection{The best uniform exponent}
For positive probabilities, Hilbert's projective distance is
\[
 d_H(\rho,\eta)=\osc_i\log(\rho_i/\eta_i).
\]
It is a metric on the normalized positive simplex. The following specialization records the exact effect of the Fisher term on classical projective contraction \cite{GaubertQu2015}.

\begin{proposition}[Optimal uniform two-solution exponent]\label{prop:hilbert}
Let $w_*=\min_{i\ne j}w_{ij}$, including zero weights for missing edges. For any two positive solutions of \eqref{eq:flow} with the same parameter $\lambda\ge0$, put $D(t)=d_H(\rho_t,\eta_t)$. Then
\begin{equation}\label{eq:hilbert-dissipation}
 D^+D(t)\le-\lambda D(t)-4w_*\sinh(D(t)/2)
 \le-(\lambda+2w_*)D(t).
\end{equation}
Consequently $D(t)\le e^{-(\lambda+2w_*)t}D(0)$. The exponent $\lambda+2w_*$ cannot be increased in an estimate with prefactor one that holds for all positive initial probabilities and every $t\ge0$. In particular, for every noncomplete graph one has $w_*=0$, so the optimal uniform prefactor-one exponent is exactly $\lambda$; the extra graph-dependent contraction in \cref{prop:state-contraction} instead uses information about the initial states.
\end{proposition}
\begin{proof}
Put $z_i=\log(\rho_i/\eta_i)$ and $D=\osc z$. Subtract the logarithmic equations for the two solutions to get
\begin{equation}\label{eq:ratio-evolution}
 \dot z_i=\sum_kw_{ik}\frac{\eta_k}{\eta_i}(e^{z_k-z_i}-1)
 -\lambda z_i+c(t),
\end{equation}
where $c(t)$ is independent of $i$. For an active maximum-ratio vertex $i$ and minimum-ratio vertex $j$, every heat term contributes a nonnegative amount to $-\dot z_i+\dot z_j$. Keeping only the direct edge gives the loss
\[
 w_{ij}\left[\frac{\eta_j}{\eta_i}(1-e^{-D})
 +\frac{\eta_i}{\eta_j}(e^D-1)\right]
 \ge4w_{ij}\sinh(D/2),
\]
by the arithmetic--geometric mean inequality. The common scalar in \eqref{eq:ratio-evolution} cancels. Taking the maximum derivative difference over active pairs proves the first inequality in \eqref{eq:hilbert-dissipation}. The second follows from $\sinh(D/2)\ge D/2$, and Gronwall's inequality gives the exponential estimate.

To prove optimality, choose a pair $(i,j)$ with $w_{ij}=w_*$. Fix $D>0$. For $N>2$, let $\varepsilon>0$ be small and put $m=1-(N-2)\varepsilon$. For $N=2$, put $m=1$ and omit the other coordinates. Define
\[
 \eta_i=\frac{m}{1+e^{D/2}},\quad
 \eta_j=\frac{me^{D/2}}{1+e^{D/2}},\quad
 \rho_i=\eta_j,\quad\rho_j=\eta_i,
\]
and $\rho_k=\eta_k=\varepsilon$ for $k\notin\{i,j\}$. Both vectors are positive probabilities. Their log ratios are $D/2$ at $i$, $-D/2$ at $j$, and zero elsewhere, so the active pair is unique. The direct-edge loss attains the arithmetic--geometric mean bound exactly. The remaining losses tend to zero as $\varepsilon\downarrow0$, for fixed $D$. Therefore
\[
 -\left.\frac{\dd}{\dd t}d_H(\rho_t,\eta_t)\right|_{t=0}
 \longrightarrow\lambda D+4w_*\sinh(D/2).
\]
An estimate with prefactor one and exponent $\beta$ valid for all positive initial probabilities would force this derivative loss to be at least $\beta D$. Letting first $\varepsilon\downarrow0$ and then $D\downarrow0$ gives $\beta\le\lambda+2w_*$.
\end{proof}

\subsection{Refinement using the amplitudes of the two trajectories}
\begin{proposition}[State-dependent additional contraction]\label{prop:state-contraction}
Let $\rho_t,\eta_t$ be two positive solutions, $D(t)=d_H(\rho_t,\eta_t)$, and
\[
 S(t)=\min\{R(\rho_t),R(\eta_t)\}.
\]
Then
\begin{equation}\label{eq:state-contraction}
 D^+D(t)\le-\lambda D(t)-e^{-S(t)}F_G(D(t)).
\end{equation}
For any nonnegative upper envelope $\overline S(t)\ge S(t)$,
\begin{equation}\label{eq:state-envelope}
 D(t)\le\log\left[1+(e^{D(0)}-1)
 \exp\left(-\lambda t-c_G\int_0^t e^{-\overline S(s)}\dd s\right)\right].
\end{equation}
\end{proposition}
\begin{proof}
First set $S_\eta=R(\eta_t)$, so $\eta_k/\eta_i\ge e^{-S_\eta}$ for every $i,k$. At an active pair in \eqref{eq:ratio-evolution}, factor this lower bound out of each nonnegative heat loss. The remaining expression is precisely of the affine form \eqref{eq:affine-loss}, with $e^D$ in place of $e^r$. The proof of \cref{thm:amplitude} gives
$D^+D\le-\lambda D-e^{-S_\eta}F_G(D)$. Interchanging the two solutions gives the same estimate with $R(\rho_t)$, proving \eqref{eq:state-contraction}.

Apply $F_G(D)\ge c_G(1-e^{-D})$ and set $U=e^D-1$. Since $e^DD\ge e^D-1$,
\[
 D^+U\le-[\lambda+c_Ge^{-\overline S(t)}]U.
\]
Integration yields \eqref{eq:state-envelope}.
\end{proof}

For $\lambda>0$, \cref{cor:amplitude-envelopes} supplies an integrable envelope $\overline S(t)$ tending to zero. Hence
$\int_0^t e^{-\overline S(s)}\dd s=t-O(1)$, with the bounded correction depending on the initial amplitudes. When $c_G>0$, \eqref{eq:state-envelope} therefore gives a long-time exponent $\lambda+c_G$ with an initial-state-dependent prefactor. This does not contradict the optimal prefactor-one exponent in \cref{prop:hilbert}.

\section{Examples, distinctions, and further questions}\label{sec:examples}

\subsection{Finite curvature need not be nonnegative}
The two-state model already separates boundedness from positivity. Let its edge weight be $w>0$ and write $\rho=(p,1-p)$. Put
\[
 q(p)=w\theta(p,1-p)+\lambda p(1-p).
\]
The metric is $\dd p^2/q(p)$. For $H(p)=p\log(2p)+(1-p)\log(2(1-p))$, the Hessian-to-metric ratio is
\[
 q(p)H''(p)+\tfrac12q'(p)H'(p).
\]
Indeed the one-dimensional Christoffel symbol is $-q'/(2q)$, and division of the covariant Hessian by $1/q$ gives the displayed expression. Define the pointwise two-state curvature quotient by $k_{\lambda,w}^{(2)}(p)=q(p)H''(p)+\tfrac12q'(p)H'(p)$. With $x=|\log((1-p)/p)|$, elementary substitution yields
\begin{equation}\label{eq:two-state}
 k_{\lambda,w}^{(2)}(p)
 =w\left(1+\frac{\sinh x}{x}\right)
 +\lambda\left(1-\frac{x\tanh(x/2)}2\right),
\end{equation}
with the continuous value at $x=0$. For any fixed $x$ satisfying $x\tanh(x/2)>2$, the coefficient of $\lambda$ is negative. Thus \eqref{eq:two-state} is negative for sufficiently large $\lambda$, although \cref{thm:classification} guarantees a finite global lower bound for every fixed $\lambda>0$ on this graph.

Conversely, \cref{cor:singular,thm:entropy-decay} show that graphs of diameter at least three can have entropic curvature $-\infty$ at every positive Fisher parameter while their entropy flow still converges exponentially. This is an explicit separation of a global geometric lower bound from a dissipation property of the selected flow.

\subsection{What the classification does and does not determine}
The diagram of implications proved in this paper is
\[
 \begin{gathered}
 \text{two-step reachability of every boundary support}\\[-2pt]
 \Updownarrow\\[-2pt]
 \diam(G)\le2
 \quad\Longleftrightarrow\quad
 \kappa_\lambda(G)>-\infty\ (\lambda>0)
 \quad\Longleftrightarrow\quad c_G>0.
 \end{gathered}
\]
These equivalences concern finite curvature and positive instantaneous amplitude damping. They do not classify the sign of the best curvature constant for each parameter. Likewise, the Fisher parameter improves the uniform damping in \eqref{eq:flow} but can destroy a global Hessian lower bound at the boundary.

Several finer questions remain separate. On a fixed diameter-two graph, one may seek the best finite lower bound, the exact initial interval of nonnegative curvature, and all states and potential directions attaining the first zero. For the four-cycle, the question of the exact reaction strength at which the global Hessian first loses nonnegativity is more delicate than determining whether the lower bound is finite. For the dynamical estimates, pointwise optimality of $F_G$ and dimension-uniform optimality of $a(r)$ do not establish optimality of the entropy prefactor along complete trajectories of a fixed graph. Finally, extension of the interior Hessian results to the metric completion, and analogous classification for a nonuniform reversible reference probability, require their own arguments.

\section*{Acknowledgment of AI assistance}
ChatGPT (OpenAI) was used to assist with literature searches, the development and checking of mathematical derivations, and manuscript drafting and editing. The authors are responsible for the accuracy and integrity of the mathematical statements, proofs, references, and final text.

\end{document}